\documentclass[12pt]{article}
\usepackage{amsfonts,amssymb,amsmath,titling,titlesec,cancel,amsthm,color, graphicx,hyperref}
\usepackage{verbatim,tikz}
\usepackage{epstopdf}
\usepackage{graphicx}
\usepackage{pifont}
\usepackage{color}
\usepackage{tkz-graph}
\usepackage{hyperref}

\AppendGraphicsExtensions{.gif}

\titleformat{\section}{\large\bfseries}{\thesection.}{1em}{}

\newtheorem{theorem}{Theorem}[section]
\newtheorem{lemma}[theorem]{Lemma}
\newtheorem{proposition}[theorem]{Proposition}
\newtheorem{corollary}[theorem]{Corollary}

\newtheorem{conjecture}[theorem]{Conjecture}

\title{Acyclic Subgraphs of Planar Digraphs}
\author{Noah Golowich\thanks{Research Science Institute, Massachusetts Institute of Technology, Cambridge MA.}~ and David Rolnick\thanks{Dept.~of Mathematics, Massachusetts Institute of Technology, Cambridge MA. Email: \href{mailto:drolnick@mit.edu}{drolnick@mit.edu}}}
\date{\today}

\begin{document}
\maketitle
\begin{abstract}
An acyclic set in a digraph is a set of vertices that induces an acyclic subgraph.  In 2011, Harutyunyan conjectured that every planar digraph on $n$ vertices without directed 2-cycles possesses an acyclic set of size at least $3n/5$. We prove this conjecture for digraphs where every directed cycle has length at least 8. More generally, if $g$ is the length of the shortest directed cycle, we show that there exists an acyclic set of size at least $(1 - 3/g)n$.
\end{abstract}
\section{Introduction}
A proper {\it vertex coloring} of an undirected graph $G$ partitions the vertices into independent sets. It is natural to try to reformulate this notion for directed graphs (digraphs).  An {\it acyclic set} in a digraph is a set of vertices whose induced subgraph contains no directed cycle. The {\it acyclic chromatic number} of a digraph $D$, denoted $\chi_A(D)$, is the minimal number of acyclic sets into which the vertices of $D$ may be partitioned. In this paper, we consider \emph{oriented graphs}, which are digraphs such that at most one edge connects any pair of vertices.

Although recent results \cite{acyclic_aharoni_2008, circular_bokal_2004, two_harut_2012, digraph_keevash_2013} suggest that the acyclic chromatic number behaves similarly to the undirected chromatic number, much still remains to be learned. 
Bokal et al.~\cite{circular_bokal_2004} credit \v Skrekovski with the conjecture that all oriented planar graphs have acyclic chromatic number at most 2.
\begin{conjecture}[\cite{circular_bokal_2004}]
\label{conj:skrekovski}
Every oriented planar graph is acyclically 2-colorable.
\end{conjecture}
Bokal et al.~\cite{circular_bokal_2004} showed that all oriented planar graphs are acyclically 3-colorable. One approach to Conjecture \ref{conj:skrekovski} has been to look for lower bounds on the size of the largest acyclic set of a planar oriented graph. Borodin \cite{borodin_acyclic_1979} showed that there must exist an acyclic set of size at least $2n/5$, where $n$ is the number of vertices. Harutyunyan and Mohar \cite{harut_planar_2014} ask whether, in any planar oriented graph, there exists an acyclic set of size at least $n/2$. Note that this would follow immediately from Conjecture \ref{conj:skrekovski}.
%Albertson \cite{albertson_problems_2002} has also suggested an easier problem than Conjecture \ref{conj:skrekovski}:
%\begin{conjecture}[Albertson, 2002]
%\label{conj:albertson}
%Every oriented planar graph contains a subset of at least half its vertices which induces an acyclic subgraph.
%\end{conjecture}
%If Conjecture \ref{conj:skrekovski} were true, then Conjecture \ref{conj:albertson} would follow immediately. 
Harutyunyan \cite{harut_brooks_2011} recently conjectured an even stronger bound on the maximum size of an acyclic set:
\begin{conjecture}[\cite{harut_brooks_2011}]
%What is the correct citation for this?
\label{conj:harut}
Every oriented planar graph on $n$ vertices contains an acyclic set of size at least $3n/5$.
\end{conjecture}

Finding the largest size of an acyclic set in a digraph is equivalent to finding a set of vertices of minimum size which has a non-empty intersection with each directed cycle. In \cite{generalized_jain_2005}, Jain et al. investigated the applications of this problem in deadlock resolution.

 The {\it digirth} of a directed graph is the length of its shortest directed cycle. Recently, Harutyunyan and Mohar \cite{harut_planar_2014} proved that every oriented planar graph of digirth 5 is acyclically 2-colorable. They used an intricate vertex-discharging method to show that any minimal counterexample must contain at least one of 25 specific configurations of vertices and edges. They then demonstrated that if a digraph contains one of these configurations, the problem of finding an acyclic 2-coloring of the digraph reduces to finding an acyclic 2-coloring of a digraph with one fewer vertex, thus showing that no minimal counterexample exists. The authors posed the problem of finding a simpler approach for considering acyclic colorings in oriented planar graphs.
 
In this paper, we introduce such an approach in Theorem \ref{thm:main}, which improves known lower bounds on the largest acyclic set in oriented planar graphs of digirth at least 4. Specifically, we give a short proof of the fact that every planar digraph on $n$ vertices and of digirth $g$ possesses an acyclic set of size at least $(1-3/g)n$. We prove this by using a corollary of the Lucchesi-Younger theorem \cite{lucchesi_minimax_1978} to find an upper bound on the size of a minimum feedback arc set of an oriented planar graph. We also give a slightly stronger bound for the cases $g = 4$ and $g=5$. Our results prove Conjecture \ref{conj:harut} when $g \geq 8$.

In Section \ref{sec:proof}, we prove our main result, Theorem \ref{thm:main}. In Section \ref{sec:counter}, we describe some potential extensions of Theorem \ref{thm:main} and difficulties that arise.
\section{Bounds on the largest acyclic set}
\label{sec:proof}
In this section, we use a corollary of the Lucchesi-Younger theorem to prove Theorem \ref{thm:main}. We begin with some definitions. Given a directed graph $D$ and a subset $X$ of its vertices $V(D)$, we define $\bar{X} = V(D) \backslash X$. If every edge between $X$ and $\bar{X}$ is directed from $X$ to $\bar{X}$, then the set of such edges is called a {\it directed cut}. A {\it dijoin} is a set of edges that has a non-empty intersection with every directed cut.

The Lucchesi-Younger theorem \cite{lucchesi_minimax_1978} gives the minimum size of a dijoin of a digraph:
\begin{theorem}[Lucchesi-Younger \cite{lucchesi_minimax_1978}]
The minimum cardinality of a dijoin in a directed graph $D$ is equal to the maximum number of pairwise disjoint directed cuts of $D$.
\end{theorem}

In the case that $D$ is planar, the Lucchesi-Younger theorem has a useful corollary for the dual of $D$. Given an oriented planar graph $D$, the {\it dual} of $D$, denoted $D^{\star}$, is defined as follows. For a given planar embedding of $D$, construct a vertex of $D^\star$ within each face of $D$. For each edge $uv$ of $D$ separating faces $f$ and $g$ of $D$, a corresponding edge $f^\star g^\star\in E(D^\star)$  is drawn between vertices $f^\star$ and $g^\star$. The direction of edge $f^\star g^\star$ is defined so that as it crosses $uv$, $v$ is on the left. It is simple to verify that the graph $D^\star$ does not depend on the planar embedding of $D$.

The following well-known result establishes a bijection between the directed cycles of a planar oriented graph and the directed cuts in its dual:
%\cite{not_sure_what_to_cite}

\begin{proposition}
\label{prop:bij}
If $D$ is a planar oriented graph, then the directed cycles of $D$ are in one-to-one correspondence with the directed cuts in $D^\star$.
\end{proposition}
\begin{proof}
Pick a planar embedding of $D$ and embed $D^\star$ in the plane as defined above. Given a directed cycle $C$ in $D$, notice that all edges of $D^\star$ crossing an edge of $C$ must travel in the same direction: specifically, if $C$ is oriented clockwise, then all edges of $D^\star$ crossing $C$ point inwards, and if $C$ is oriented counterclockwise, then the edges of $D^\star$ crossing $C$ point outwards. Let $X \subset V(D^\star)$ consist of the vertices of $D^\star$ corresponding to all faces of $D$ inside $C$, and therefore the edges of $D^\star$ connecting $X$ and $V(D^\star) \backslash X$ form a directed cut.

Conversely, given a directed cut of $V(D^\star)$, we reverse the method above to obtain a directed cycle of $D$.
\end{proof}

Given a directed graph, a {\it feedback arc set} is a set of edges, the removal of which eliminates all directed cycles. We will be concerned with \emph{minimum feedback arc sets}, namely those of minimum cardinality. Proposition $\ref{prop:bij}$ establishes the following corollary of the Lucchesi-Younger theorem:

\begin{corollary}[\cite{lucchesi_minimax_1978}]
\label{cor:lucchesi}
For a planar oriented graph, the minimum size of a feedback arc set is equal to the maximum number of arc-disjoint directed cycles.
\end{corollary}

\begin{proof}
Given a planar oriented graph $D$, by the Lucchesi-Younger theorem, the minimum cardinality of a dijoin of $D^\star$ is equal to the maximum number of disjoint directed cuts of $D^\star$. Now, by Proposition $\ref{prop:bij}$, any dijoin of $D^\star$ corresponds to a unique feedback arc set of $D$ of the same size, and any set of disjoint directed cuts of $D^\star$ corresponds to a unique set of arc-disjoint directed cycles of $D$, also of the same size. This completes the proof.
\end{proof}

We also need one well-known lemma, which follows easily from Euler's formula for planar graphs.

\begin{lemma}
\label{lem:mn}
Any planar graph $G$ with $n$ vertices and $m$ edges satisfies $m \leq 3n - 6$.
\end{lemma}

We now are ready to state and prove our main theorem.

\begin{theorem}
\label{thm:main}
If $D$ is a planar digraph with digirth $g$ on $n$ vertices, then there exists an acyclic set in $G$ of size at least $n - 3n/g$. Moreover, if $g  = 4$, there exists an acyclic set of size at least $5n/12$, and if $g = 5$, there exists an acyclic set of size at least $7n/15$.
\end{theorem}

\begin{proof}
A \emph{vertex cover} of a given edge set $E$ is a set $S$ of vertices for which every edge in $E$ is incident to some vertex in $S$. Observe that if $S$ is a vertex cover of a feedback arc set in a digraph $D$, then the complement of $S$ in $D$ is an acyclic set. This is because every directed cycle must include an element of the feedback arc set.

Given a planar oriented graph $D$ of digirth $g$, let $H$ be a collection of arc-disjoint directed cycles, each of which must have length at least $g$. Thus, by Lemma \ref{lem:mn} the number of cycles in $H$ is at most $e(D)/g \leq 3n/g$, where $e(D)$ denotes the number of edges of $D$. Corollary \ref{cor:lucchesi} implies that there exists a feedback arc set with cardinality at most $3n/g$. Removing a vertex cover of this feedback arc set, we are left with an acyclic set of at least $n - 3n/g$ vertices.

%\begin{theorem}
%On a planar digraph of girth 4, there is an acyclic set of $5n/12$ vertices, and on a planar digraph of girth 5, there is an acyclic set of $8n/15$ vertices.
%\end{theorem}
For the cases $g = 4,5$, we now derive a better bound by using the greedy algorithm to obtain a smaller vertex cover of the feedback arc set. This is possible because $3n/g > n/2$, meaning that some vertices are incident to 2 or more feedback arcs. Suppose that our feedback arc set consists initially of $f$ edges, where $f\le 3n/g$.

Let $d$ equal the number of feedback arcs minus half the number of vertices. Thus, initially $d=f-n/2$.  At each step, we remove the vertex $v$ that is incident to the most feedback arcs, together with all feedback arcs incident to $v$. As long as $d>0$, each step removes one vertex and at least two feedback arcs. Such a step decreases $d$ by at least $3/2$. Let $m$ be the number of steps taken before $d\le 0$, and let $d'\le 0$ be the final value of $d$. We conclude that $$m\le \frac{f-n/2-d'}{3/2}=\frac{2f-n-2d'}{3}.$$

After $m$ steps, the number of vertices remaining is $n-m$, so the number of feedback arcs remaining is $d'+(n-m)/2$. We remove one vertex from each of these feedback arcs so that the vertices remaining at the end form an acyclic set. The total number of vertices removed is

\begin{align*}
m+d'+\frac{n-m}{2}&=\frac{n}{2}+d'+\frac{m}{2}\\ &\le \frac{n}{2}+d'+\frac{1}{2}\cdot \frac{2f-n-2d'}{3}\\ &=\frac{n+f+2d'}{3}\\ &\le \frac{n+f}{3}\\ &\le \frac{n}{3}+\frac{n}{g}.
\end{align*}

The number of vertices remaining in our acyclic set is thus at least$$\frac{2n}{3}-\frac{n}{g}.$$

This is equal to $5n/12$ for $g=4$ and $7n/15$ for $g=5$.

% Another way:
% Suppose we remove $f_i \geq 2$ vertices at each step. Then the number of steps is exactly the smallest integer $m$ such that there is a positive real number $D$ such that
% \begin{equation}
% \label{eq:nm}
% D + \frac{3n}{g} - \sum_{i = 1}^m f_i = \frac{n-m}{2}.
% \end{equation}
% Now, the total number of vertices we need to cover all the feedback arcs is $m + 3n/g - \sum_{i = 1}^m f_i$. By $(\ref{eq:nm})$, this number of vertices is at most $\frac{n+m}{2} - D$. % if equality holds in (\ref{eq:nm}), and $\frac{n+m}{2} - \frac{1}{2}$ otherwise. 
% Now, we let $f' = \frac{ \sum_{i = 1}^m f_i}{m}$, meaning that $f' \geq 2$. Thus, we can rewrite (\ref{eq:nm}) as
% \begin{equation}
% m\left(f' - \frac{1}{2}\right) = n \left( \frac{6-g}{2g} \right) - D\nonumber
% \end{equation}
% Since $m$ is the smallest integer such that there exists $D>0$ so that the above is satisfied, we have
% \begin{equation}
% m =  \frac{n\left(\frac{6-g}{2g}\right) - D}{f' - \frac{1}{2}} \leq \frac{n\left(\frac{6-g}{2g}\right)}{f' - \frac{1}{2}}\nonumber,
% \end{equation}
% and since $f' - 1/2 \geq 3/2$, we have that $m \leq n(6-g)/(3g).$ Thus, the maximum set of vertices remaining is of size at least
% \begin{equation}
% \frac{2n}{3} - \frac{n}{g}\nonumber.
% \end{equation}
\end{proof}

\section{Further directions}
\label{sec:counter}
In this section we describe methods that might be used to strengthen our main result, and difficulties that arise.

\subsection{Acyclic 2-colorings}
Harutyunyan and Mohar \cite{harut_planar_2014} asked whether there is a simple proof of the fact that planar oriented graphs of large digirth are acyclically 2-colorable. We showed above that we can find a very large acyclic set in a planar digraph of large digirth. It is natural to ask whether, given a minimum feedback arc set as in Theorem \ref{thm:main}, we can choose a vertex cover $Y$ of the feedback arc set such that $Y$ induces an acyclic subgraph. This would imply that planar digraphs of large digirth are acyclically 2-colorable, because the subgraph induced by $V(D) \backslash Y$ contains no feedback arcs and hence is acyclic.  We now show that there are planar graphs for which no such $Y$ can be found.

\begin{proposition}
\label{prop:bigraph}
There exists an oriented planar graph $D$ such that, for any minimum feedback arc set $F$ of $D$, and for any vertex cover $Y$ of $F$, there is a directed cycle in the subgraph induced by $Y$.
\end{proposition}

\begin{proof}
We show that the digraph $D$ in Figure \ref{fig:bigraph} satisfies the necessary conditions. The unique minimum feedback arc set  $\{ab, bc, ac, ed, gf, hi\}$ is indicated with dashed lines. Notice that any vertex cover of the minimum feedback arc set must contain at least 2 vertices from the central triangle $abc$, suppose $a$ and $b$ without loss of generality. We must now choose either $d$ or $e$ to cover edge $de$. However, if we choose $d$, then $abd$ is a directed 3-cycle, and if we choose $e$, then $abe$ is a directed 3-cycle. This completes the proof.
\end{proof}

%
% If we can find a minimal feedback arc set which is bipartite, that is, which does not contain any odd cycles, then partitioning $D$ into 2 acyclic sets is simple: we pick vertices from one of the vertex sets of the bipartite feedback arc set to be in color class 1, and take the rest of the vertices in $D$ to be in color class 2.  If either color class contains any cycles, then these cycles must avoid all feedback arcs, which is a contradiction.
%
%However, we can construct graphs for which the minimum feedback arc set is not bipartite, as shown in figure []. A partition of the vertices of the graph in figure[] is shown: notice that for any choice of vertices covering all of the feedback arcs from the minimal feedback arc set, these vertices can not all belong to one color class.

%It is also natural to conjecture whether in any planar graph of relatively high digirth, there is a vertex of in-degree or out-degree at most 1. However, as shown in figure[], there is a planar digraph of arbitrarily high digirth such that the in-degree and the out-degree of every vertex is at least 2.

\begin{figure}
\SetVertexNormal[Shape      = circle,
                 FillColor  = white,
                 LineWidth  = .5pt]
\SetUpEdge[lw         = 1.3pt,
           color      = black,
           labelcolor = white,
           labeltext  = red,
           labelstyle = {sloped,draw,text=blue}]
\begin{center}
\begin{tikzpicture}
\tikzset{VertexStyle/.style = {shape = circle,fill = pink,minimum size = 2pt,inner sep = 2.5pt}}

   \Vertex[x=-1.5 ,y=0]{a}
   \Vertex[x=1.5,y=0]{b}
   \Vertex[x=0,y=2.6]{c}
   \Vertex[x=-1.5,y=-4]{d}
   \Vertex[x=1.5,y=-4]{e}
   \Vertex[x=4.96,y=2]{f}
   \Vertex[x=3.4,y=4.6]{g}
   \Vertex[x=-3.4,y=4.6]{h}
   \Vertex[x=-4.96,y=2]{i}
   \Vertex[x=0,y=-.6]{j}
   \Vertex[x=1.27,y=1.6]{k}
   \Vertex[x=-1.27,y=1.6]{l}
   \Vertex[x=0,y=-1.5]{m}
   \Vertex[x=2.49,y=2.05]{n}
   \Vertex[x=-2.49,y=2.05]{o}
   \Vertex[x=0,y=-4.6]{p}
   \Vertex[x=0,y=-5.5]{q}
   \Vertex[x=4.73,y=3.6]{r}
   \Vertex[x=5.96,y=4.05]{s}
  \Vertex[x=-4.73,y=3.6]{t}
   \Vertex[x=-5.96,y=4.05]{u}

   \tikzset{EdgeStyle/.style = {->}}

   \Edge(b)(j)
   \Edge(j)(a)
   \Edge(c)(l)
   \Edge(l)(a)
   \Edge(c)(k)
   \Edge(k)(b)
   \Edge(c)(h)
   \Edge(i)(a)
   \Edge(c)(g)
   \Edge(f)(b)
   \Edge(b)(e)
   \Edge(d)(a)
   \Edge(b)(m)
   \Edge(m)(a)
   \Edge(c)(n)
   \Edge(n)(b)
   \Edge(c)(o)
   \Edge(o)(a)
   \Edge(f)(r)
   \Edge(r)(g)
   \Edge(f)(s)
   \Edge(s)(g)
   \Edge(d)(p)
   \Edge(p)(e)
   \Edge(d)(q)
   \Edge(q)(e)
   \Edge(i)(t)
   \Edge(t)(h)
   \Edge(i)(u)
   \Edge(u)(h)

   \tikzset{EdgeStyle/.append style = {bend left}}
   \Edge(e)(a)
   \Edge(b)(d)
   \Edge(c)(f)
   \Edge(g)(b)

   \tikzset{EdgeStyle/.append style = {bend right}}
   \Edge(h)(a)
   \Edge(c)(i)

\SetUpEdge[lw         = 2.5pt,
           color      = black,
           labelcolor = white,
           labeltext  = red,
           labelstyle = {sloped,draw,text=blue}]
   \tikzset{EdgeStyle/.style = {dashed,->}}
   \Edge(a)(b)
   \Edge(b)(c)
   \Edge(a)(c)
   \Edge(e)(d)
   \Edge(g)(f)
   \Edge(h)(i)

\end{tikzpicture}
\end{center}
\caption{The planar oriented graph $D$. The unique minimum feedback arc set of $D$ is shown in dashed lines. Note that the $D$ is planar since edge $ea$ can be relocated around $abeqd$, and similarly for $gb$ and $ci$.}
\label{fig:bigraph}
\end{figure}
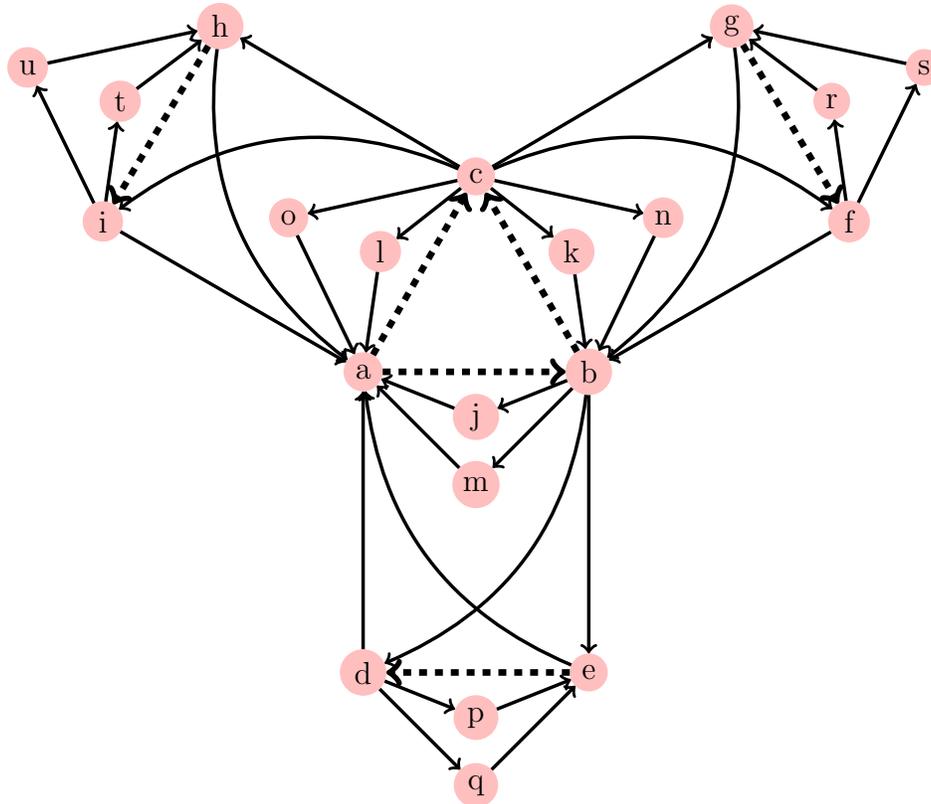

\subsection{Improving bounds for small $g$}

For small $g$, the bound we obtain in Theorem \ref{thm:main} on the maximum acyclic set of an oriented graph on $n$ vertices is less than $n/2$, and it is natural to ask whether we can improve upon this bound. In Propositions \ref{prop:triangle} and \ref{prop:square}, we now present two classes of planar oriented graphs, of digirth 3 and 4, respectively, for which the minimum size of a set of feedback arcs is significantly greater than $n/2$. Therefore, removing one vertex from each feedback arc yields an acyclic set of size less than $n/2$. Although some feedback arcs may overlap in certain vertices, as we demonstrate in the proof of Theorem \ref{thm:main}, applying the greedy algorithm still does not show the existence of an acyclic set of size at least $n/2$. Hence
these classes of graphs suggest that the methods used in Theorem \ref{thm:main} are unlikely to prove significantly stronger bounds for digirth 3 and 4.

Given a digraph $D$, let $f(D)$ denote the size of a minimum set of feedback arcs.

\begin{proposition}
\label{prop:triangle}
There exists an infinite family of digraphs $\mathcal{D}^3$, such that for all $D\in \mathcal{D}^3$:
\begin{itemize}
\item $D$ is planar, with digirth 3.
\item $f(D)=|V(D)|-2$.
\end{itemize}
\end{proposition}
\begin{proof}
We define $D_i\in \mathcal{D}^3$ recursively. Let $D_0$ be a directed 3-cycle, and for $i\ge 0$, construct $D_{i+1}$ as follows from a planar embedding of $D_i$. Pick some face $abc$ of $D_i$ that forms a directed 3-cycle. Construct vertices $d,e,f$ within $abc$, with edges as shown in Figure \ref{fig:triangle}.

Observe that $|V(D_{i+1})|=|V(D_i)|+3$ and $f(D_{i+1})=f(D_i)+3$. Since $|V(D_0)|=3$ and $f(D_0)=1$, it follows by induction that $f(D_i)=|V(D_i)|-2$ for every $i$.  By construction, $D_i$ is planar and has digirth 3.
\end{proof}

\begin{figure}
\SetVertexNormal[Shape      = circle,
                 FillColor  = white,
                 LineWidth  = .5pt]
\SetUpEdge[lw         = 1.3pt,
           color      = black,
           labelcolor = white,
           labeltext  = red,
           labelstyle = {sloped,draw,text=blue}]
\begin{center}
\begin{tikzpicture}
\tikzset{VertexStyle/.style = {shape = circle,fill = pink,minimum size = 2pt,inner sep = 2.5pt}}

   \Vertex[x=0,y=4]{a}
   \Vertex[x=3.464,y=-2]{b}
   \Vertex[x=-3.464,y=-2]{c}
   \Vertex[x=0,y=-1]{d}
   \Vertex[x=-0.866,y=0.5]{e}
   \Vertex[x=0.866,y=0.5]{f}

 \tikzset{EdgeStyle/.style = {->}}

   \Edge(a)(b)
   \Edge(b)(c)
   \Edge(c)(a)
   \Edge(f)(a)
   \Edge(a)(e)
   \Edge(d)(b)
   \Edge(b)(f)
   \Edge(e)(c)
   \Edge(c)(d)
   \Edge(d)(e)
   \Edge(e)(f)
   \Edge(f)(d)
   
   \end{tikzpicture}
\end{center}
\caption{Directed octahedral pattern for generating graphs of digirth 3 with large minimum feedback set.}
\label{fig:triangle}
\end{figure}
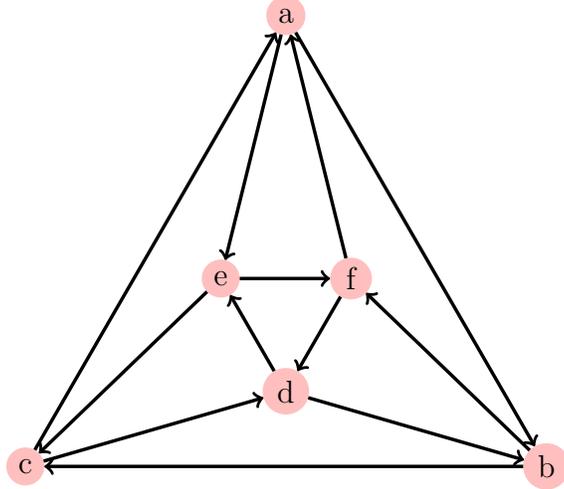

%\begin{corollary}
%For $D\in \mathcal{D}^3$, removing one vertex from each of a minimum set of feedback arcs yields an acyclic set with only two elements.
%\end{corollary}

\begin{proposition}
\label{prop:square}
There exists an infinite family of digraphs $\mathcal{D}^4$, such that for all $D\in \mathcal{D}^4$:
\begin{itemize}
\item $D$ is planar, with digirth 4.
\item $f(D)=\frac{5}{8}|V(D)|-\frac{3}{2}$.
\end{itemize}
\end{proposition}
\begin{proof}
As with $\mathcal{D}^3$, we define $D_i\in \mathcal{D}^4$ recursively. Let $D_0$ be a directed 4-cycle, and for $i\ge 0$, construct $D_{i+1}$ as follows from a planar embedding of $D_i$. Pick some face $abcd$ of $D_i$ that forms a directed 4-cycle. Construct vertices $e,f,g,h,i,j,k,l$ within $abcd$, with edges as shown in Figure \ref{fig:square}.

Observe that $|V(D_{i+1})|=|V(D_i)|+8$ and $f(D_{i+1})=f(D_i)+5$. Since $|V(D_0)|=4$ and $f(D_0)=1$, it follows by induction that $f(D_i)=\frac{5}{8}|V(D_i)|-\frac{3}{2}$ for every $i$.  By construction, $D_i$ is planar and has digirth 4.
\end{proof}

\begin{figure}
\SetVertexNormal[Shape      = circle,
                 FillColor  = white,
                 LineWidth  = .5pt]
\SetUpEdge[lw         = 1.3pt,
           color      = black,
           labelcolor = white,
           labeltext  = red,
           labelstyle = {sloped,draw,text=blue}]
\begin{center}
\begin{tikzpicture}
\tikzset{VertexStyle/.style = {shape = circle,fill = pink,minimum size = 2pt,inner sep = 2.5pt}}

   \Vertex[x=-4,y=4]{a}
   \Vertex[x=4,y=4]{b}
   \Vertex[x=4,y=-4]{c}
   \Vertex[x=-4,y=-4]{d}
   \Vertex[x=0,y=3]{e}
   \Vertex[x=3,y=0]{f}
   \Vertex[x=0,y=-3]{g}
   \Vertex[x=-3,y=0]{h}
   \Vertex[x=-1,y=1]{i}
   \Vertex[x=1,y=1]{j}
   \Vertex[x=1,y=-1]{k}
   \Vertex[x=-1,y=-1]{l}
   
 \tikzset{EdgeStyle/.style = {->}}

   \Edge(a)(b)
   \Edge(b)(c)
   \Edge(c)(d)
   \Edge(d)(a)
   \Edge(a)(e)
   \Edge(e)(b)
   \Edge(b)(f)
   \Edge(f)(c)
   \Edge(c)(g)
   \Edge(g)(d)
   \Edge(d)(h)
   \Edge(h)(a)
   \Edge(e)(i)
   \Edge(i)(h)
   \Edge(h)(l)
   \Edge(l)(g)
   \Edge(g)(k)
   \Edge(k)(f)
   \Edge(f)(j)
   \Edge(j)(e)
   \Edge(j)(i)
   \Edge(k)(j)
   \Edge(l)(k)
   \Edge(i)(l)
   
   \end{tikzpicture}
\end{center}
\caption{Directed cuboctahedral pattern for generating graphs of digirth 4 with large minimum feedback set.}
\label{fig:square}
\end{figure}
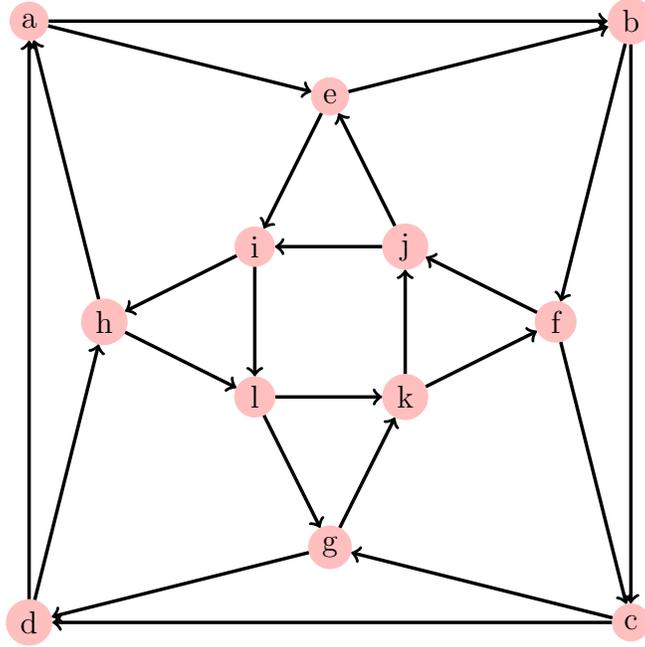

We have defined a class of digraphs $\mathcal{D}^3$ by recursive addition of an octahedral pattern, and a class $\mathcal{D}^4$ by recursive addition of a cuboctahedral pattern. It is possible to define a similar class $\mathcal{D}^5$, containing planar digraphs of digirth 5, by recursive addition of an icosidodecahedral pattern. However, the resulting relation on $f(D)$ falls short of $|V(D)|/2$:$$f(D)=\frac{11}{25}|V(D)|-\frac{6}{5}.$$Thus, for $D\in \mathcal{D}^5$, removing a vertex from each of a minimum set of feedback arcs yields an acyclic set of size greater than $|V(D)|/2$. We cannot say whether the methods of Theorem \ref{thm:main} may hold for general planar graphs of digirth 5.

% \begin{figure}
% \label{fig:pentagon}
% \SetVertexNormal[Shape      = circle,
%                  FillColor  = white,
%                  LineWidth  = .5pt]
% \SetUpEdge[lw         = 1.3pt,
%           color      = black,
%           labelcolor = white,
%           labeltext  = red,
%           labelstyle = {sloped,draw,text=blue}]
% \begin{center}
% \begin{tikzpicture}
% \tikzset{VertexStyle/.style = {shape = circle,fill = pink,minimum size = 2pt,inner sep = 2.5pt}}

%   \Vertex[x=4 ,y=0]{a}
%   \Vertex[x = 1.236, y=3.804]{b}
%   \Vertex[x=-3.236,y=2.351]{c}
%   \Vertex[x=-3.236,y=-2.351]{d}
%   \Vertex[x=1.236,y=-3.804]{e}
%   \Vertex[x=2.118,y=1.532]{f}

%   \tikzset{EdgeStyle/.style = {->}}

%   \Edge(a)(b)
%   \Edge(b)(c)
%   \Edge(c)(d)
%   \Edge(d)(e)
%   \Edge(e)(a)
%   \Edge(d)(a)
%   \Edge(b)(d)
%   \Edge(a)(f)
%   \Edge(f)(b)

% \end{tikzpicture}
% \end{center}
% \caption{A planar oriented graph with a unique minimum feedback arc set (in bold and blue) and such that any vertex cover of this set induces a directed cycle. Note that the graph is planar since edge $ea$ can be relocated around $abeqd$, and similarly for $gb$ and $ic$.}
% \end{figure}

\section{Acknowledgements}
The authors would like to thank Jacob Fox for calling our attention to this problem and Tanya Khovanova for helpful discussions and review. We would also like to thank the MIT Department of Mathematics, the Center for Excellence in Education, and the Research Science Institute for their support of this research.
%\begin{thebibliography}
\bibliographystyle{plain}
\bibliography{biblio.bib}

\begin{thebibliography}{1}

\bibitem{acyclic_aharoni_2008}
Ron Aharoni, Eli Berger, and Ori Kfir.
\newblock Acyclic systems of representatives and acyclic colorings of digraphs.
\newblock {\em Journal of Graph Theory}, pages 177--189, 2008.

\bibitem{lucchesi_minimax_1978}
Cl\' audio Lucchesi and D.~H. Younger.
\newblock A minimax theorem for directed graphs.
\newblock {\em Journal of London Mathematical Society}, 2:369--374, 1978.

\bibitem{circular_bokal_2004}
Drago Bokal, Ga\v{s}per Fijav\v{z}, Martin Juvan, P.~Mark Kayll, and Bojan
  Mohar.
\newblock The circular chromatic number of a digraph.
\newblock {\em Journal of Graph Theory}, pages 227--240, 2004.

\bibitem{borodin_acyclic_1979}
O.~V. Borodin.
\newblock On acyclic colorings of planar graphs.
\newblock {\em Discrete Mathematics}, 25:~211--236, 1979.

\bibitem{harut_brooks_2011}
Ararat Harutyunyan.
\newblock {\em Brooks-type results for coloring of digraphs}.
\newblock PhD thesis, Burnaby, British Columbia, 2011.

\bibitem{two_harut_2012}
Ararat Harutyunyan and Bojan Mohar.
\newblock Two results on the digraph chromatic number.
\newblock {\em Discrete Mathematics}, 312(10):~1823--1826, 2012.

\bibitem{harut_planar_2014}
Ararat Harutyunyan and Bojan Mohar.
\newblock Planar digraphs of digirth 5 are 2-colorable.
\newblock {\em arXiv}, 1401.2213, 2014.

\bibitem{generalized_jain_2005}
Kamal Jain, Mohammad~Taghi Hajiaghayi, and Kunal Talwar.
\newblock The generalized deadlock resolution problem.
\newblock {\em 32nd International Colloquium on Automata, Languages and
  Programming}, 3580:~853--865, 2005.

\bibitem{digraph_keevash_2013}
Peter Keevash, Zhentao Li, Bojan Mohar, and Bruce Reed.
\newblock Digraph girth via chromatic number.
\newblock {\em SIAM Journal on Discrete Mathematics}, 27(2):~693--696, 2013.

\end{thebibliography}
%\end{thebibliography}

\end{document}